\documentclass[12pt]{amsart}
\usepackage{latexsym}
\usepackage{amsmath}
\usepackage{amssymb}
\usepackage{amscd}
\usepackage{amsthm}
\usepackage{xypic}
\xyoption{curve}
\usepackage{ifthen} 
\usepackage{hyperref} 
\usepackage{graphicx}
\theoremstyle{plain}
 \newtheorem{theorem}{Theorem}[section]
 \newtheorem{proposition}[theorem]{Proposition}
 \newtheorem{lemma}[theorem]{Lemma}
 \newtheorem{corollary}[theorem]{Corollary}
 
\theoremstyle{definition}
 \newtheorem{definition}[theorem]{Definition}

\theoremstyle{remark}
 \newtheorem{remark}[theorem]{Remark}

\makeindex
\begin{document}
\title[Affine groups]{Rationality of quotients by linear actions of affine
groups}
\author[Bogomolov]{Fedor Bogomolov}
\author[B\"ohning]{Christian B\"ohning} 
\author[Graf v. Bothmer]{Hans-Christian Graf von Bothmer}
\maketitle
{\em  We feel honored to dedicate this article to our friend, colleague and teacher Fabrizio
Catanese on the occasion of his 60th birthday.}
\newcommand{\PP}{\mathbb{P}} 
\newcommand{\QQ}{\mathbb{Q}} 
\newcommand{\ZZ}{\mathbb{Z}} 
\newcommand{\CC}{\mathbb{C}} 
\newcommand{\rmprec}{\wp}
\newcommand{\rmconst}{\mathrm{const}}
\newcommand{\xycenter}[1]{\begin{center}\mbox{\xymatrix{#1}}\end{center}} 
\newboolean{xlabels} 
\newcommand{\xlabel}[1]{ 
                        \label{#1} 
                        \ifthenelse{\boolean{xlabels}} 
                                   {\marginpar[\hfill{\tiny #1}]{{\tiny #1}}} 
                                   {} 
                       } 
\setboolean{xlabels}{false} 

\

\begin{abstract}
Let $G=\mathrm{SL}_n (\CC ) \ltimes \CC^n$ be the (special) affine group. In this paper we study the representation theory of $G$ and in
particular the question of rationality for $V/G$ where $V$ is a generically free $G$-representation. We show that the answer to this
question is positive (Theorem \ref{tA}) if the dimension of $V$ is sufficiently large and $V$ is indecomposable. For two-step extensions $0 \to S
\to V \to Q \to 0$ with $S$, $Q$ completely reducible we explicitly characterize those whose rationality cannot be obtained by the coarse methods
presented here (Theorem \ref{tB}).
\end{abstract}
\section{Introduction} \xlabel{sIntroduction}
The well-known rationality problem in invariant theory asks whether $V/G$ is always rational if $V$ is a linear representation of a
connected linear algebraic group $G$ over $\CC$. This seems to be extremely difficult in general. However, it becomes a little more
accessible if the unipotent radical of $G$ is large in a certain sense, of which Miyata's Theorem is the first example: if the action of $G$
on $V$ can be made triangular, $V/G$ is rational. We will give further evidence for the previous viewpoint in this paper by studying
generically free quotients $V/G$ where $G = \mathrm{SL}_n (\CC ) \ltimes \CC^n$ is the special affine group. In fact, if $V$ is indecomposable
 and of
sufficiently large dimension, these quotients are always rational, cf. Theorem \ref{tA} and Theorem \ref{tB} below. Some sort of indecomposability
assumption is really needed as there are families of decomposable arbitrarily large generically free $G$-representations for which a proof of
rationality amounts to proving stable rationality of level $1$ for all generically free $\mathrm{SL}_n (\CC )$-representations, which is expected
to be a hard problem , cf. Remark
\ref{rExceptionsWeCannotDo}. 
\
One should also note that many rationality questions for reductive groups reduce to parabolic subgroups by the method of taking sections for
the action.
\
We remark that the methods of this paper apply in principle more generally to the affine groups $\mathrm{GL}_n (\CC )\ltimes \CC^n$,
$\mathrm{Sp}_n (\CC )\ltimes \CC^n$, $\mathrm{SO}_n (\CC )\ltimes \CC^n$ and other affine extensions of semisimple groups, or even to other nilpotent extensions of reductive groups where one knows stable rationality of some
level for the reductive part such as jet groups. But we felt that treating all these cases uniformly might have rendered the presentation
less transparent, and that it would be better to focus on a sample case to illustrate the methods.\\
\textbf{Acknowledgments.} During the work the first author was supported
by the NSF grant DMS 0701578. He also wants
to thank Korean Institute of Advanced Study and
Institut des Hautes \'{E}tudes Scientifiques for
 hospitality and support during the work on the paper. 
The second and third author were supported by the German Research Foundation 
(Deutsche Forschungsgemeinschaft (DFG)) through 
the Institutional Strategy of the University of G\"ottingen. The second author also wants to thank IH\'{E}S for
hospitality during a stay when work on the subject was done.
\section{Preliminaries} \xlabel{sAffineBundles}
We begin by recalling some standard facts and conventions which we will need in the sequel. We work over the complex numbers throughout.
\begin{itemize}
\item[(A)] 
Let $G$ be a connected linear algebraic group. $G$ is an extension
\begin{gather*}
1 \to U \to G \to R \to 1 \, ,
\end{gather*}
where $U$ is the unipotent radical of $G$, and $R$ is the reductive part (representations of it are completely reducible). Thus $U$ is
nilpotent as a group (the descending central series terminates in the trivial group), and all elements
$u$ of
$U$ are unipotent, i.e. $n=1-u$ is nilpotent. $G$ is then a semidirect product $G= R \ltimes U$ (Levi decomposition) and the reductive part
may be written
\[
R = (T \times S )/ C
\]
where $T$ is a torus, $S$ is semisimple, and $C$ a finite central subgroup. A (finite dimensional) $G$-representation $V$ has a
Jordan-H\"older filtration
\begin{gather}\label{formulaFiltrationOne}
(0) \subset V_0 \subset V_1 \subset \dots \subset V_{l-1} \subset V_l=V
\end{gather}
by $G$-invariant subspaces such that the quotient $V_{i+1}/V_{i}$ is a completely reducible $G$-representation (so $U$ acts trivially on
$V_{i+1}/V_{i}$) and maximal with that property.
\item[(B)]
If $\Gamma$ is any linear algebraic group ($G$, $U$, ...), and $W$ a $\Gamma$-representation, then $W/\Gamma$ will always denote the quotient
in the sense of Rosenlicht in the sequel, i.e. a birational model of $\CC (W)^{\Gamma}$.
\item[(C)]
The groups $G=\mathrm{SL}_n (\CC )$ resp. $G=\mathrm{SAff}_n (\CC )=\mathrm{SL}_n (\CC ) \ltimes \CC^n$ are special (every \'{e}tale locally
trivial principal $G$-bundle is Zariski locally trivial), so for a generically free
$G$-representation $V$, $V/G$ is stably rational of level $\dim G$ ($=n^2-1$ resp. $n^2-1+n$).
\item[(D)]
Let $U=\CC^n$ be the $n$-dimensional additive group (which occurs as unipotent radical e.g. in the affine group $\mathrm{SL}_n (\CC )\ltimes \CC^n$).
\begin{lemma}\xlabel{lAffineBundles}
Let $0 \to A \to B \to C \to 0$ be an exact sequence of $U$-representations such that the $U$ action on $A$ and $C$ is trivial. Then $B/U \to C/U =C$ is
birationally a vector bundle over $C$.
\end{lemma}
\begin{proof}
After choosing a section $\sigma_0$ of the projection $B\to C$, $B$ becomes a trivial vector bundle $A\oplus C$ over $C$ with zero section $\sigma_0$. The
$U$-orbit of $\sigma_0$ inside $B$ is then generically a vector subbundle of $B$: an element $t \in \CC^n$ acts on the fibre $A\times \{ c_0 \}$ as $(a, \:
c_0) \mapsto (a + t(c_0), \: c_0)$ via translations. Thus $U\cdot \sigma_0$ is a family
of vector subspaces in each fibre, trivialized by the sections $e_1\cdot \sigma_0, \dots , e_n \cdot \sigma_0$ where $e_1, \dots , e_n$ is a basis of
$\CC^n$, over some open set in $C$ (where the dimension of the space of translations they span is the generic one). Then $B/U$ may be identified with the
quotient bundle
$B/U\cdot
\sigma_0$.
\end{proof}
\end{itemize}
\section{Representations of affine groups} \xlabel{sAffine}
In this section we review the representation theory of the affine group, see also \cite{Specht} on this.  
Let $\mathrm{SAff}_n (\CC ) = \mathrm{SL}_n (\CC ) \ltimes \CC^n$ be the $n$-dimensional special affine group. We will write
$U=\CC^n$ sometimes, to avoid confusion, if we consider it as a subgroup of $\mathrm{SAff}_n (\CC )$. Elements of
$\mathrm{SAff}_n (\CC )$ can be written in matrix form as
\begin{gather*}
\left( \begin{array}{cc} A & v \\ 0 & 1 \end{array} \right)
\end{gather*}
where $A \in \mathrm{SL}_n (\CC )\subset \CC^{n\times n}$, $v$ is a vector in $\CC^{n\times 1}$, and $0 \in \CC^{1\times n}$. Thus we can
write elements $g\in \mathrm{SAff}_n (\CC )$ as $g = (A, \: v)$, and matrix multiplication yields
\begin{gather}\label{formulaSemidirect}
(A, 0 ) \cdot (\mathrm{id} , v ) = (\mathrm{id}, Av) \cdot (A, 0) \, .
\end{gather}
Let $V$ be an $N$-dimensional representation of $\mathrm{SAff}_n (\CC )$. In a suitable basis, the image of the additive subgroup $\CC^n$ (of
pairs $(\mathrm{id}, v)$) under $\varrho \, :\, \mathrm{SAff}_n (\CC ) \to \mathrm{Aut} (V) = \mathrm{GL}_N (\CC )$ is contained in the
unipotent subgroup $U_N$ of upper triangular $N\times N$ matrices with ones on the diagonal, which is an affine space. Thus
\begin{gather}\label{formulaPolynomialNature}
\varrho (v) = \sum_{|\alpha | \le d } F_{\alpha } v^{\alpha }
\end{gather}
where $\alpha = (\alpha_1, \dots , \alpha_n) \in \mathbb{N}^n$  is a multiindex, $|\alpha | = \alpha_1 + \dots + \alpha_n$, $v^{\alpha } =
v_1^{\alpha_1} \cdot \dots \cdot v_n^{\alpha_n}$ as usual, and $F_{\alpha }$ is in $\mathrm{Mat}_{N\times N} (\CC )$.
\
The representation $V$ has \emph{two} natural filtrations:
\begin{gather}\label{formulaFiltration}
0 \subset V_0 \subset V_1 \subset \dots \subset V_{l-1} \subset V_l =V ,
\end{gather}
and
\begin{gather}\label{formulaFiltration2}
0 \subset V_0' \subset V_1' \subset \dots \subset V_{l-1}' \subset V_l' =V ,
\end{gather}
which are defined inductively as follows: $Q_i := V_{i}/V_{i-1}$, for $i=0, \dots , l$ (we put $V_{-1}=0$), is the  maximal completely reducible
subrepresentation of $V/V_{i-1}$; and $Q_{l-j}' := V_{l-j}' /V_{l-j-1}'$, for $j=0, \dots , l$, is the maximal completely reducible quotient
representation of $V_{l-j}'$. Thus 
\begin{gather*}
V=Q_0 \oplus Q_1 \oplus \dots \oplus Q_l = Q_0' \oplus \dots \oplus Q_l'
\end{gather*}
as $\mathrm{SL}_n (\CC )$-representations (this is sometimes called the \emph{semisimplification} of $V$).
\begin{remark}\xlabel{rRelationFiltrations}
The two methods of filtering a representation of the affine group are related to duality as follows: if $V$ has a filtration of type
(\ref{formulaFiltration}) with quotients $Q_i = V_{i}/V_{i-1}$, then the dual $W:=V^{\vee }$ has a filtration
\[
0\subset W_0' \subset W_1' \subset \dots \subset W_{l-1}' \subset W_l'=W
\]
of type (\ref{formulaFiltration2}) with $W_{l-j}' /W_{l-j-1}'=:Q_{l-j}' = Q_j^{\vee }$.  
\end{remark}
\
We will first consider filtrations of type (\ref{formulaFiltration}) in this section, and unless explicitly stated otherwise, the term
\emph{filtration} will mean \emph{filtration of type} (\ref{formulaFiltration}).
\
Since $\rho (\mathrm{exp}(u) ) = \mathrm{exp}
(d\varrho_e (u))$ for the linearization $d \varrho_e \, : \, \CC^n \to \mathrm{End} (V)$, we see that $d$ in formula
\ref{formulaPolynomialNature} can be chosen equal to $l$. More precisely, $\varrho (v)$ is represented by some $N\times N$ block matrix
\begin{gather}\label{formulaMatrixRepresentative}
\left( \begin{array}{ccccc}    
\mathrm{Id}_{q_0} & N_{01} & N_{02} & \dots & N_{0l}\\
   0         & \mathrm{Id}_{q_1} & N_{12} & \dots & N_{1l}\\
  \vdots     & \ddots            & \ddots & \dots & \vdots\\
0 &  0 & \dots      &   \mathrm{Id}_{q_{l-1}} & N_{l-1, l}\\
0 & 0 &  0 & \dots & \mathrm{Id}_{q_l}       
\end{array}\right)
\end{gather}
where $q_i = \dim Q_i$, and $N_{ij}$ is a $q_i \times q_j$-matrix depending on $v$, $N_{ij} = N_{ij}(v)$, and $N_{ij}(v)$ \emph{is a
polynomial in $v$ of total degree} $\le j-i$.
\
Clearly $\mathrm{SAff}_n (\CC )$ is a subgroup of $\mathrm{SL}_{n+1} (\CC )$ in the natural way, and every $\mathrm{SL}_{n+1} (\CC
)$-representation yields an $\mathrm{SAff}_n (\CC )$-representation by restriction. In particular, $\mathrm{Sym}^l (\CC^{n+1})^{\vee }$
yields an $\mathrm{SAff}_n (\CC )$-representation with a filtration
\[
0 \subset S_0 \subset \dots \subset S_l ,
\]
such that $S_{i}/S_{i-1} = \mathrm{Sym}^i (\CC^n)^{\vee }$ as $\mathrm{SL}_n (\CC )$-representations (this is the action of $\mathrm{SAff}_n
(\CC )$ on affine functions of degree less than or equal to $l$).
\begin{proposition}\xlabel{pStructureAffRepresentations}
Every representation $V$ of the group $\mathrm{SAff}_n (\CC )$ with a filtration as in \ref{formulaFiltration} is a subrepresentation of 
\[
V_0 \otimes \mathrm{Sym}^l (\CC^{n+1})^{\vee } 
\]
with filtration induced from
\[
0 \subset V_0 \subset V_0 \otimes S_1 \subset \dots V_0 \otimes S_{l-1} \subset V_0 \otimes S_l = V_0 \otimes \mathrm{Sym}^l
(\CC^{n+1})^{\vee } \, .
\]
\end{proposition}
\begin{proof}
Let $v, w \in U=\CC^n$, $A\in\mathrm{SL}_n (\CC )$ and $x \in V$. Then
\[
\varrho (v) (\varrho (A) x) = \varrho (A) (\varrho (A^{-1}v) x) = \sum_{|\alpha | \le l} \varrho (A) F_{\alpha }(x) (A^{-1} v)^{\alpha }
\] 
and 
\[
\varrho (v) (\varrho (w) x) = \sum_{|\alpha | \le l} F_{\alpha }(x) (v+w)^{\alpha }
\]
which means that there is an $\mathrm{SAff}_n (\CC )$-equivariant map
\begin{gather*}
V \to (V_0 \oplus V_1/V_0 \oplus \dots \oplus V_{l}/V_{l-1}) \otimes \mathrm{Sym}^l (\CC^{n+1})^{\vee }\\
x \mapsto f_x
\end{gather*}
where $f_x$ is the (affine) polynomial function on $U=\CC^n$ (with coefficients in $V$ \emph{viewed as} $Q_0 \oplus \dots \oplus Q_l$ \emph{now as}
$\mathrm{SAff}_n (\CC )$-\emph{module}!) given by
\[
f_x (v) = \varrho (v) (x)\, .
\]
There is an $\mathrm{SAff}_n (\CC )$-equivariant projection
\[
(V_0 \oplus V_1/V_0 \oplus \dots \oplus V_{l}/V_{l-1}) \otimes \mathrm{Sym}^l (\CC^{n+1})^{\vee } \to V_0 \otimes \mathrm{Sym}^l
(\CC^{n+1})^{\vee }\]
It gives us an $\mathrm{SAff}_n (\CC )$-equivariant map
\[
\iota : V \to V_0 \otimes \mathrm{Sym}^l(\CC^{n+1})^{\vee }
\]
and it remains to check injectivity for this map. Injectivity follows from the assumption that the filtration \ref{formulaFiltration} is
such that the $Q_{i+1}$ are \emph{maximal} completely reducible subrepresentations of $V/V_{i}$ in each step (we will actually only be using
that $Q_0$ is the maximal completely reducible submodule of $V$ in the proof of the Proposition and the full assertion in the proof of the
following Corollary): for assume to the contrary that injectivity fails. By
$\mathrm{SL}_n (\CC )$-equivariance this is equivalent to saying that there exists an
$\mathrm{SL}_n (\CC )$-irreducible summand $S$ of $V=Q_0 \oplus \dots \oplus Q_l$ which is mapped to $0$ under $\iota$. Then the
$\mathrm{SAff}_n (\CC )$-span $\bar{S}$ of
$S$ in
$V$ is contained entirely in $Q_I \oplus \dots \oplus Q_l$ for some $I\ge 1$. But $\bar{S}$ contains a completely reducible submodule (a
minimal $\mathrm{SAff}_n (\CC )$-submodule) which intersects $Q_0$ trivially. This contradicts the maximality of $Q_0$.
\end{proof}
\begin{corollary}\xlabel{cBlocks}
For a representation $V$ of $\mathrm{SAff}_n (\CC )$ as above we have for $j \ge i$ that 
\[
Q_j \subset Q_i \otimes \mathrm{Sym}^{j-i} (\CC^n)^{\vee }
\]
as $\mathrm{SL}_n (\CC )$-representations. 
\end{corollary}
\begin{proof}
It suffices to show that $Q_i \subset Q_0 \otimes \mathrm{Sym}^i (\CC^n)^{\vee }$ follows from the previous Proposition
\ref{pStructureAffRepresentations}. The general case follows by replacing $V$ by $V/V_{i-1}$. We have the $\mathrm{SAff}_n (\CC
)$-equivariant map
\[
\iota\, :\, V \to V_0 \otimes \mathrm{Sym}^{l} (\CC^{n+1})^{\vee }
\]
in particular an induced $\mathrm{SL}_n (\CC )$-equivariant map
\[
Q_i \to Q_0 \otimes \mathrm{Sym}^i (\CC^n)^{\vee }
\]
which is the restriction of the previous map to $Q_i$ composed with the $\mathrm{SL}_n (\CC )$-equivariant projection. We just have to prove
it is nonzero on every irreducible summand $S$ of $Q_i$. If to the contrary it is zero on $S$ this would mean that $Q_0 \oplus Q_1 \oplus
\dots \oplus Q_{i-1} \oplus S$ and also the $\mathrm{SAff}_n (\CC )$-submodule $\bar{S}$ generated by it is mapped under $\iota$ to the
$\mathrm{SAff}_n (\CC )$-submodule $V_0 \otimes\mathrm{Sym}^{i-1} (\CC^{n+1})^{\vee }$ of $V_0\otimes \mathrm{Sym}^{l} (\CC^{n+1})^{\vee }$.
But this would mean that $\bar{S}$ has a filtration
\[
0 \subset \bar{S}_0 \subset \dots \subset \bar{S}_{i-1}=\bar{S}
\]
with completely reducible quotients which contradicts the fact that the $Q_{i+1}$ are the maximal completely reducible subrepresentations of
$V/V_i$ in each step (here we are using this fact in its full strength). 
\end{proof}
Let us now consider filtrations of type (\ref{formulaFiltration2}) and dualize the statements in Proposition \ref{pStructureAffRepresentations} and
Corollary \ref{cBlocks}.
\begin{proposition}\xlabel{pDualStructureTheorem}
Let $V$ be a representation of $\mathrm{SAff}_n (\CC )$ with a filtration of type (\ref{formulaFiltration2}):
\[
0 \subset V_0' \subset V_1' \subset \dots \subset V_{l-1}' \subset V_l' =V
\]
with $Q_{l-j}' := V_{l-j}' /V_{l-j-1}'$, for $j=0, \dots , l$, the maximal completely reducible quotient
representation of $V_{l-j}'$. Then $V$ is a quotient of $Q_l' \otimes \mathrm{Sym}^l (\CC^{n+1})$ and for $i\le j$
\[
Q_i' \subset Q_j' \otimes \mathrm{Sym}^{j-i} (\CC^n) \, .
\]
\end{proposition}
\begin{proof}
The maximal completely reducible subrepresentation of $V^{\vee}$ is $(Q_l')^{\vee }$, so by Proposition \ref{pStructureAffRepresentations}, $V^{\vee}$
is a subrepresentation of $(Q_{l}')^{\vee} \otimes \mathrm{Sym}^{l} (\CC^{n+1})^{\vee }$ from which the first assertion follows. By Remark
\ref{rRelationFiltrations} and Corollary \ref{cBlocks}, one has for $t\ge s$
\[
(Q_{l-t}')^{\vee } \subset (Q_{l-s}')^{\vee } \otimes \mathrm{Sym}^{t-s} (\CC^n)^{\vee }
\]
from which the second assertion follows putting $i=l-t$, $j=l-s$ and dualizing. 
\end{proof}
\section{Minimal subvarieties of Severi-Brauer varieties} \xlabel{sSeveriBrauer}
References for the theory of Severi-Brauer varieties are \cite{Art}, \cite{Gi-Sza}, \cite{Sa99}. We start with a few recollections. A
Severi-Brauer variety
$P$ over a field
$K$ is one that becomes isomorphic to projective space $P_{\bar{K}} \simeq \mathbb{P}_{\bar{K}}$ over the algebraic closure of $K$. Thus a
fibration $X \to Y$ which is generically a projective bundle in the \'{e}tale topology over $Y$ gives rise to a Severi-Brauer variety over
$K=\CC (Y)$. If $A$ is an Azumaya (central simple) algebra of degree $n$ over $K$ (i.e. $A\otimes_K \bar{K}\simeq \mathrm{Mat}_{n\times n}
(\bar{K})$), then the set of all minimal (i.e. dimension
$n$) right ideals $I$ of $A$ is a closed subvariety $\PP_A$ of $\mathrm{Grass}(n, \: A)$ defined by the conditions that $I$ is a right ideal.
This is a Severi-Brauer variety as can be seen from the fact that for $A\simeq \mathrm{Mat}_{n\times n} (K) = \mathrm{End}(V)$ the right
ideals of dimension $n$ are in bijective correspondence with $\PP (V)$ by associating to a one dimensional subspace $l$ in the
$n$-dimensional
$K$-vector space $V$ those $f \in \mathrm{End}(V)$ with image contained in $l$. Conversely, any Severi-Brauer variety arises in this way
since both isomorphism classes of Severi-Brauer varieties of dimension $n-1$ over $K$ and isomorphism classes of degree $n$ Azumaya algebras
over $K$ are classified by the nonabelian Galois cohomology set $H^1(\mathrm{Gal}(\bar{K}/K), \: \mathrm{PGL}_n (\bar{K})) = H^1 (K, \:
\mathrm{PGL}_n (\bar{K}))$ (note $\mathrm{PGL}_n (\bar{K}) = \mathrm{Aut} (\mathrm{Mat}_{n\times n} (\bar{K}))$, so the automorphism groups
over $\bar{K}$ of $K$-forms of projective space and $K$-forms of matrices are the same). The inductive limit $H^1 (K,\: \mathrm{PGL}_{\infty
})$ of the sets $H^1 (K,\: \mathrm{PGL}_n)$ via the maps $H^1 (K, \: \mathrm{PGL}_{n}) \to H^1 (K, \: \mathrm{PGL}_{mn})$ (diagonal
embedding) carries a natural group structure induced by the tensor product $\mathrm{PGL}_n \times \mathrm{PGL}_m \to \mathrm{PGL}_{m\times
n}$. Then there is the isomorphism $H^1 (K,\: \mathrm{PGL}_{\infty}) \simeq \mathrm{Br}(K)$ with the Brauer group $\mathrm{Br}(K)$ of $K$,
and each Severi-Brauer variety $P$ has its class $[P]\in \mathrm{Br}(K)$. We need the following two lemmas linking the birational geometry
and algebra of Severi-Brauer varieties.
\begin{lemma}\xlabel{lSeveriBrauerBundle}
If $D$ is a division algebra over $K$ and if $A = \mathrm{Mat}_{r\times r} (D)$, then the associated Severi-Brauer variety $P_A$ over $K$ is
birational to the projectivisation of a vector bundle on $P_D$.
\end{lemma}
\begin{proof}
Let $e$ be an idempotent with $eAe=D$ and let $P_A \to P_D$ be the map that sends a right ideal $I$ in $A$ to $eIe$. After the base change
to $\kappa= K (P_D)$, the generic point of $P_D$, this map can be identified with the projection $\pi\, :\, \PP (V^{\oplus r}_{\kappa }) =
P_A\otimes
\kappa 
\to
\PP (V_{\kappa })= P_D \otimes \kappa$ onto a summand. Note that $\kappa $ is a splitting field for both $D$ and $A$. The generic fibre of
$P_A\to P_D$ is the preimage of the point defined by the generic point in $P_D\otimes \kappa $ under $\pi$.
\end{proof}
\begin{lemma}\xlabel{lSeveriBrauerStable}
If two Severi-Brauer varieties $P_1$ and $P_2$ over $K$ are stably birationally isomorphic over $K$, then $[P_1]$ and $[P_2]$ generate the
same subgroup of $\mathrm{Br}(K)$ and conversely. More precisely, if $P_1$ and $P_2$ are both of rank $r$, then 
\[
P_1 \times_K \PP^r \simeq  P_2 \times_K \PP^r
\]
where $\simeq$ denotes birational equivalence.
\end{lemma}
\begin{proof}
This follows from Amitsur's theorem (Theorem 5.4.1 of \cite{Gi-Sza}) that the kernel of $\mathrm{Br}(K) \to \mathrm{Br}(K(P))$ for a
Severi-Brauer variety
$P$ over the function field $K$ is generated by the class of $[P]$, and from the fact that $\mathrm{Br}(K(P)) \to \mathrm{Br}(K(P)(t))$ is
injective for an indeterminate $t$. See \cite{Gi-Sza}, Remark 5.4.3, for the converse.
\end{proof}
\begin{lemma}\xlabel{lProjectiveRepresentatives}
Suppose $S_1$ is some rank $r$ Severi-Brauer variety over $K$. Then the class $n [S_1] \in \mathrm{Br}(K)$ is also representable by a
Severi-Brauer variety of rank $r$. 
\end{lemma}
\begin{proof}
By hypothesis, $S_1$ corresponds to an Azumaya algebra $A$ of degree $r+1$ over $K$. The index of $A^{\otimes n}$ divides the index of $A$
(\cite{Gi-Sza}, Prop. 4.5.8). Recall that the index of an Azumaya algebra is the degree of the unique division algebra in its Brauer equivalence
class. In other words, there is an Azumaya algebra of the same degree as $A$ which represents the class of $n [S_1]$.
\end{proof}
\begin{proposition}\xlabel{pRatBundles}
Let
\[
0 \to V_0 \to V \to Q_1 \to 0
\]
be a two-step filtration of a generically free $G$-representation $V$, where $G$ is the special affine group
$\mathrm{SAff}_n (\CC )$ (this means here simply that $V_0$ is some completely reducible subrepresentation, and $Q_1$ is a completely reducible
quotient). Assume that
$\mathrm{SL}_n (\CC ) / (\ZZ /m
\ZZ )$ acts generically freely on the quotient $Q_1$ for some $m \mid n$. Suppose moreover that $\dim V_0 \ge n^2 + 2n$. Then $V/G$ is
rational.
\end{proposition}
\begin{proof}
Look at the fibre product diagram
\[
\begin{CD}
X @>>> V' /\mathrm{SL}_n (\CC )\\
@VVV    @VVV \\
(Q_1 \oplus \CC^n )/\mathrm{SL}_n (\CC ) @>>> Q_1/\mathrm{SL}_n (\CC )
\end{CD}
\]
where $V' = V/U$. Here $X$ is a vector bundle over both $V'/\mathrm{SL}_n (\CC )$ and $(Q_1 \oplus \CC^n)/\mathrm{SL}_n (\CC )$. This
follows from the no-name lemma of \cite{Bo-Ka} because the action of $\mathrm{SL}_n (\CC )$ on both $V'$ and $Q_1 \oplus \CC^n$ is
generically free. This means that if we divide out homotheties in the fibres and consider the Severi-Brauer varieties
\begin{gather*}
S_1 \, :\, \PP (V')/\mathrm{SL}_n (\CC ) \to Q_1 /\mathrm{SL}_n (\CC ) , \\ S_2 \, : \, (Q_1 \oplus \PP (\CC^n ))/\mathrm{SL}_n (\CC ) \to
Q_1 /\mathrm{SL}_n (\CC ) \, ,
\end{gather*}
then $S_1$ and $S_2$ are stably equivalent Severi-Brauer varieties. By Lemma \ref{lSeveriBrauerStable} $S_1$ and $S_2$ generate the same
subgroup of $\mathrm{Br}(K)$ where $K = \CC (Q_1 / \mathrm{SL}_n (\CC ))$. By Lemma \ref{lProjectiveRepresentatives} the class of $S_1$ is
also represented by some $\PP^{n-1}$-bundle $S'$. By Lemma \ref{lSeveriBrauerStable} $S'$ and $S_2$ are stably equivalent of level $n-1$.
Moreover, $S_2$ is stably rational of level $n^2$. Furthermore, by Lemma \ref{lSeveriBrauerBundle}, $S_1$ is birational to a vector bundle
over $S'$ provided its rank is bigger than $n-1$. If the rank of $S_1$ is bigger than $n^2 + n-1$ we consequently get rationality of $S_1$
from the stable rationality of level $n^2$ of $S'$. The latter follows because $S_2$ is stably rational of level $n^2$ and $S'$ and $S_2$
are stably equivalent of level $n-1$, and $n^2 \ge  n-1$. Together with rationality for $S_1$ we obtain of course also rationality of
$V/G$. Finally the rank of $S_1$ is bigger or equal to $\dim V_0 -n - 1$ (the subgroup $\CC^n$ still acts via
translations). Hence the assertion.  
\end{proof}
\section{Rationality for 2-step extensions}
\begin{definition}\xlabel{dBadRepresentations}
We call an $\mathrm{SL}_n (\CC )$-representation $W$ \emph{bad} if no group isogeneous to $\mathrm{SL}_n (\CC )$ (i.e. obtained from
$\mathrm{SL}_n (\CC )$ by quotienting by a finite subgroup of its centre) acts generically freely on $W$. If $n > 9$ the irreducible bad
representations are
\[
\Lambda^2 (\CC^n), \: S^2 (\CC^n ), \: \CC^n, \: \CC , \: \mathrm{Ad}_0 = \Sigma^{(2, 1, \dots , 1, \: 0)} (\CC^n )
\] 
or one of the duals (where $\mathrm{Ad}_0$ is the trace zero part of the adjoint representation and $\Sigma$ denotes the Schur functor). Every bad
representation is of course a direct sum of these irreducible bad ones.\\
If $W$ is not bad we will also say that $W$ is \emph{good}.
\end{definition}
\begin{remark}\xlabel{rBadFinite}
For given $n$ there are -up to addition of trivial summands $\CC$- only finitely many bad representations. For every bad
$\mathrm{SL}_n (\CC )$-representation is a direct sum of irreducible bad ones, and for each of the irreducible bad representations $R$ other than $\CC$ (they
are all listed in an appendix table of \cite{Po-Vi}) it is true that there is a $t$ such that $R^t = R \oplus \dots \oplus R$ ($t$-times) is
good.
\end{remark}
\begin{theorem}\xlabel{tB}
If $V$ is a generically free $G$-representation with a filtration of type (\ref{formulaFiltration2}) of length $l=2$, then we can write $V$ as $V= V_1
\oplus W$ where $V_1$ has a length $2$ filtration
\[
0 \to S \to V_1 \to Q \to 0
\]
with $S\subset Q\otimes\CC^n$, $Q\subset S\otimes (\CC^n)^{\vee }$, $W$ is an $\mathrm{SL}_n (\CC )$-representation, and $V_1$ does not split off
another $\mathrm{SL}_n (\CC )$-representation. Then $V/G$ is rational if 
\begin{itemize}
\item[(A)]
there exists an $\mathrm{SL}_n (\CC )$-equivariant decomposition $W=W_1 \oplus W_2$ such that 
\begin{itemize}
\item
$Q\oplus W_2$ is good \emph{and}
\item
$\dim (S\oplus W_1) \ge n^2 + 2n$.
\end{itemize}
\item[(B)]
\emph{or} $Q$ contains $\ge n^2-1$ copies of $\CC$.
\end{itemize}
\end{theorem}
\begin{proof}
The assertion (A) of the Theorem follows from Proposition \ref{pRatBundles}. For part (B) remark that 
if $Q$ does contain more than $n^2-1$ copies of $\CC$ then there is a natural
$(G, \: \mathrm{SL}_n (\CC ))$-section for the action of $G$ on $V$. 
Namely writing $Q := Q'\oplus \CC$ the action of $U$ on $V$ gives an $\mathrm{SL}_n (\CC
)$-equivariant map
\[
U \simeq \CC^n \to \mathrm{Hom} (Q' \oplus \CC , S) \to \mathrm{Hom} (\CC, \: S)\simeq S
\]
and the map of $\CC^n$ to $S$ cannot be zero, hence must be an inclusion, since otherwise the action of $U$ on the summand $\CC$ in
$Q$ would be trivial and hence $S$ would not be maximal with the property of being a completely reducible subrepresentation of $V_1$. 
So $S = S' \oplus \CC^n$, and as $(G, \: \mathrm{SL}_n (\CC ))$-section we take $S' \oplus Q \oplus W$.
Now $Q = Q'' \oplus m \CC$ with $m \ge n^2-1$. By assumption $\mathrm{SL}_n (\CC )$ operates generically freely on $S'\oplus Q\oplus W$  and therefore
also on $S' \oplus Q''\oplus W$. But then
$(S' \oplus Q \oplus W)/ \mathrm{SL}_n (\CC )$ is generically a rank $m$ vector bundle over 
$(S' \oplus Q''\oplus W)/ \mathrm{SL}_n (\CC )$ which is in turn stably rational of level $n^2-1$. 
It follows that $V$ is also rational in this case.
\end{proof}
\begin{remark}\xlabel{rExceptionsWeCannotDo}
There are cases of generically free $G$-representations $V$ which do not satisfy the hypotheses of Theorem \ref{tB} and for which a proof of 
rationality for $V/G$ seems to be a difficult problem: take $V = \CC^{n+1} \oplus W$, $W$ an irreducible generically free $\mathrm{SL}_n (\CC
)$-representation. Then, as in the preceding proof, $V$ has a $(G, \:\mathrm{SL}_n (\CC ))$-section $\CC \oplus W$. Thus, if we could prove
rationality for $V/G$, we would obtain stable rationality of level $1$ for $W/\mathrm{SL}_n (\CC )$, which is expected to be a hard problem in
general.
\end{remark}
\begin{definition}\xlabel{dExceptionalTwoStep}
We call a $G$-representation $V$ with a length two filtration \emph{exceptional} if it does not satisfy the hypotheses of Theorem \ref{tB}, hence
we cannot conclude that $V/G$ is rational immediately by the methods here. More precisely, if we write $V=V_1 \oplus W$ as above this means that
either $V$ is not generically free or it is so but 
\begin{itemize}
\item[(1)]
for all decompositions $W=W_1 \oplus W_2$, $Q\oplus W_2$ is bad or $\dim S\oplus W_1 < n^2+2n$
\item[(2)]
\emph{and} $Q$ contains $< n^2-1$ copies of $\CC$.
\end{itemize}
\end{definition}
For later use, and to characterize exceptional two-step representations we prove the following technical
\begin{lemma}\xlabel{lPassageSLtoG}
Let
\[
0 \to S \to V \to Q \to 0
\]
be an exact sequence of $G$-representations with $S$ 
the \emph{maximal} completely reducible submodule of $V$. Assume that 
\begin{itemize}
\item[(a)]
the maximal completely reducible quotient of $Q$ is good as $\mathrm{SL}_n (\CC )$-representation. 
\item[(b)]
the maximal completely reducible subrepresentation of $Q$ is not one of the following finitely many (bad) $\mathrm{SL}_n (\CC )$-representations:
\[
R_1 = \CC^n \; \mathrm{or}\; R_2 = \Lambda^2 (\CC^n)^{\vee }\; \mathrm{or} \; R_3 =\CC \oplus \dots \oplus \CC \oplus (\CC^n)^{\vee } \oplus
\dots \oplus (\CC^n)^{\vee }
\]
(at most $(n-1)$-summands in total in $R_3$).
\end{itemize}
Then
$V$ is generically free as
$G$-representation. 
\end{lemma}
\begin{proof} 
From (a) one gets that $V$ is generically free for $\mathrm{SL}_n (\CC )$. We argue by contradiction and suppose that $V$ is not generically 
$G$-free.  This means that for general $v\in V$ there exists an element $(A,
\: t) \in G =\mathrm{SL}_n (\CC ) \ltimes \CC^n$ \emph{with} $t\neq 0$ such that
\begin{gather*}
\left( \begin{array}{cc} \varphi_1 (A) & \psi (A, \: t) \\ 0 & \varphi_2 (A, t)  \end{array} \right) \left( \begin{array}{c} v_1 \\ v_2
\end{array} \right) = \left( \begin{array}{c} v_1 \\ v_2  \end{array} \right)
\end{gather*}
where $v_1$ resp. $v_2$ are the components of $v$ w.r.t. the $\mathrm{SL}_n (\CC )$-equivariant splitting $V = S \oplus Q$ and $\varphi_1
(A)$, $\varphi_2 (A, t)$ and $\psi (A, \: t)$ denote the components of $(A, \: t)$ relative to the representation $V$ in $\mathrm{End}(S)$,
$\mathrm{End}( Q)$ and $\mathrm{Hom} (Q, \: S)$. 
\
Since a group isogeneous to $\mathrm{SL}_n (\CC )$ acts generically freely on the maximal completely reducible quotient of $Q$, we obtain
that $A$ is multiplication by some root of unity $\zeta_A$. Thus there exists an element $v^0_1 \in S$ such that for almost all $v_2 \in Q$  we
get the equation
\[
\psi (A, \: t) v_2 = (1 - \varphi_1 (A) ) v_1^0\, .
\]
We consider the variety
\begin{gather*}
M := \left\{ (A, \: t , \: v_2) \in \CC^{n\times n } \times (\CC^n\backslash \{0 \}) \times Q \, \mid \,  \psi (A, \: t) v_2 = (1 - \varphi_1 (A) )
v_1^0  
\right\}.
\end{gather*}
By what was said above, the projection to $Q$ of $M$ is dominant, and since $M$ is the union of
components $M_{\zeta }$ corresponding to different $\zeta$'s, one of them will dominate $Q$, so that we can consider $\zeta = \zeta_0$ and also 
$A=A_0$ as fixed. Now consider the projection
\begin{gather*}
p\, :\, M_{\zeta_0 } \to \PP (\CC^n) = \PP (U), \\
(A_0, \: t, \: v_2 ) \mapsto [t] \, .
\end{gather*}
For a $[t_0]$ in the image of $p$ with maximal fibre dimension, we get 
\[
\dim p^{-1} (t_0) \ge \dim M_{\zeta_0 } - (n-1) \ge \dim Q -(n-1)
\]
and since for $w\in p^{-1}(t_0)$ fixed and general $v_2 \in p^{-1}(t_0)$ we have $\psi (A_0, \: t_0 )(v_2 - w) = 0$, we get
\[
\dim \mathrm{ker } (\psi (A_0, \: t_0) ) \ge \dim p^{-1} (t_0 ) \ge \dim Q - (n-1) \, .
\]
We may view $\psi (A_0, \cdot )$ as a family of 
$\mathrm{SL}_n (\CC )$-equivariant homomorphisms
\[
\psi_i (A_0, \cdot ) \, : \, \mathrm{Sym}^i (\CC^n ) \to \mathrm{Hom} (Q_i, \: S)
\]
for $i$ between $1$ and the filtration length $L$ of $V$ (thus there is a natural grading). 
By $\mathrm{SL}_n (\CC )$-equivariance 
\[
\dim \mathrm{ker } (\psi (A_0, \: t) ) \ge  \dim Q - (n-1)
\]
\emph{for all} $t$.
We now work on $\PP^{n-1} = \PP
(U )$, and, since $H^0 (\PP(U ), \: \mathcal{O}(1) ) = U^{\vee }$, we may view $\psi$ as giving rise to maps of vector bundles
\[
\begin{CD}
0 @>>> \mathrm{ker} (\psi ) (A_0, \cdot ) @>>> \bigoplus_{i=1}^L Q_i \otimes \mathcal{O}(-i+1) @>{\psi (A_0, \cdot ) }>> S \otimes \mathcal{O}(1)
\end{CD}
\]
where $\mathrm{ker}(\psi ) (A_0 , \cdot )$ is a vector bundle by $\mathrm{SL}_n (\CC )$-equivariance. We will suppress $A_0$ from the
notation in the sequel and restate our basic inequality in the form
\begin{gather}\label{FormulaBasicInequality}
\mathrm{rk} (\mathrm{ker}(\psi )) + \dim \PP (U) \ge \dim Q
\end{gather} 
Now factor $\psi$:
\[
\begin{CD}
\bigoplus_{i=1}^L Q_i \otimes \mathcal{O}(-i+1) @>{\alpha }>> \mathrm{im} (\psi ) @>{\beta }>> S \otimes \mathcal{O}(1) \, .
\end{CD}
\]
Since by (\ref{FormulaBasicInequality}) we have $n-1 \ge \mathrm{rk} (\mathrm{im} (\psi ))$ one can only have
\begin{gather*}
\mathrm{im} (\psi ) = \mathcal{T}_{\PP^{n-1}} (k) \intertext{or} \mathrm{im} (\psi ) = \Omega^{1}_{\PP^{n-1}} (k) \intertext{or}
\bigoplus_{i=1}^l \mathcal{O}(k_i), \quad l \le n-1 \, .
\end{gather*}
We will now narrow down the number of possibilities for $\mathrm{im} (\psi )$ even further. The map induced by the bundle map $\psi$ on the
$H^0$-level corresponds to the map $\CC^n \to \mathrm{Hom}( Q_1, \: S)$ where $Q_1$ is the maximal completely reducible subrepresentation of
$Q$, thus is nonzero and hence 
\begin{gather}\label{FormulaRestrictionOfPossibilities}
H^0 (\mathrm{im} (\psi )) \neq 0 \quad \mathrm{and} \quad H^0 ((\mathrm{im} (\psi))^{\vee }(1)) \neq 0 \, . 
\end{gather}
Using (\ref{FormulaRestrictionOfPossibilities}), we find that $\mathrm{im} (\psi )$ can only be one of the
following:
\begin{gather*}
\mathcal{T} (-1), \quad \Omega^1 (2) , \quad    (X_0 \otimes \mathcal{O}) \oplus (X_1 \otimes \mathcal{O}(1)) \oplus \dots \oplus (X_{L-1}
\otimes
\mathcal{O}(-(L-1))\, .
\end{gather*}
Here $X_0, \dots , X_{L-1}$ are some $\mathrm{SL}_n (\CC )$-representations since we are looking at homogeneous vector bundles.
We have $H^0 (\Omega (2)) = \Lambda^2 (\CC^n )^{\vee }$, $H^0 (\mathcal{T} (-1)) = \CC^n$, and all $X$'s must be direct sums of trivial
representations $\CC$ since $\mathrm{rk} (\mathrm{im} (\psi )) \le n-1$. We will argue that none of these cases can actually occur under the
hypotheses of the Lemma. For in the induced sequence 
\[
\begin{CD}
Q_1\otimes \mathcal{O} @>{\alpha }>> \mathrm{im} (\psi ) @>{\beta }>> S \otimes \mathcal{O}(1) \, .
\end{CD}
\]
the arrows $\alpha$ and $\beta$ are equivariant, hence it follows in each of the three cases above that $Q_1$ contains $R_1=\CC^n$ 
or $R_2=\Lambda^2
(\CC^n )^{\vee }$ or a direct sum
\[
R_3 = \CC \oplus \dots \CC \oplus (\CC^n)^{\vee } \oplus \dots \oplus ( \CC^n )^{\vee }
\]
(at most $(n-1)$-copies in total since $\mathrm{rk} (\mathrm{im} (\psi )) \le n-1$) \emph{and} the arrow $\alpha$ is induced by the
arrows in the Euler sequence 
\[
\begin{CD}
0 @>>> \mathcal{O}(-1) @>>> \CC^n \otimes \mathcal{O}Ê@>>> \mathcal{T}_{\PP^{n-1}} (-1) @>>> 0 
\end{CD}
\]
and the identity map $\mathcal{O}\to \mathcal{O}$. Hence there are two possibilities:
\begin{itemize}
\item
$Q_1$ is equal to $R_1$, $R_2$ or $R_3$ which is impossible by assumption (b).
\item
$Q_1\otimes\mathcal{O}$ splits as $(R_j \oplus Q_1')\otimes \mathcal{O}$, and $Q_1'$ is a nonzero subrepresentation of $Q_1$ which is in the
 kernel of
$\psi$. This contradicts the hypothesis that $S$ is
maximal with the property of being completely reducible inside $V$.
\end{itemize}
\end{proof}
\begin{corollary}\xlabel{cCharExceptional}
Let $V=V_1\oplus W$, $0\to S \to V_1 \to Q\to 0$, be an exceptional two step representation (notation as in Theorem \ref{tB}). If $Q$ contains $<
n^2-1$ summands of $\CC$, then there are only finitely many possibilities for $Q$ and $S$ for any fixed $n$.
\end{corollary}
\begin{proof}
Suppose $V$ is exceptional because the action of $G$ on it, hence on $V_1$, is not generically free. Then by Lemma \ref{lPassageSLtoG}, $Q$ must be
bad and by Remark \ref{rBadFinite}, since it contains $< n^2-1$ summands of $\CC$, there are only finitely many possibilities for $Q$, hence since $S
\subset Q\otimes \CC^n$ also finitely many possibilities for $S$.
\
If $V$ is generically free, but (1) of Definition \ref{dExceptionalTwoStep} is satisfied, then the dimension of $S$ must be $< n^2+2n$ or $Q$ must be
bad, which in view of $S\subset Q\otimes \CC^n$ and $Q \subset S \otimes (\CC^n)^{\vee }$ again limits both $S$ and $Q$ to finitely many possibilities.
\end{proof}
\begin{corollary}\xlabel{cNotGenericallyFree}
If $V$ is a $G$-representation of large enough filtration length $l$ (with respect to filtration types (\ref{formulaFiltration}) or 
(\ref{formulaFiltration2})), then $V$ is generically free.
\end{corollary}

\begin{proof}
There are two cases.
\\
(1) $Q_l$ in the type (\ref{formulaFiltration})-filtration for $V$ is good. Then $Q_{l-1} \to V/V_{l-2} \to Q_l$ is a quotient of $V$ which is
generically free by Lemma \ref{lPassageSLtoG}, so $V$ is generically free.\\

(2) $Q_l$ is bad, a sum of irreducible bad representations. The $G$-span of $Q_l$ is contained in $Q_l \otimes \mathrm{Sym}^l (\CC^{n+1})$. We
use the following
immediate consequence of the Littlewood-Richardson rule:
\begin{quote}
For an irreducible $\mathrm{SL}_n (\CC )$-representation $W =\Sigma^{(\lambda_1, \dots , \lambda_n)} (\CC^n)$, $\lambda_1 \ge
\dots \ge \lambda_n\ge 0$ a non-increasing sequence of non-negative integers, put $\lambda (W) := \lambda_1 (W) -\lambda_2 (W)$. Then,
if $U$ is an irreducible summand of $W\otimes\mathrm{Sym}^{k} (\CC^n)$, we have
\[
\lambda (U) \ge k - \lambda_1 (W)\, .
\]
\end{quote}
Since $\lambda$ and $\lambda_1$ are bounded on irreducible bad representations, we see that if $l$ is sufficiently large, then we can assume the following properties for 
$V$:
\begin{itemize}
\item[(a)]
$V$ is an extension $V_1\to V\to V_2$ of $G$-representations with $V_1$ generically free for the action of $\mathrm{SL}_n (\CC )$ (we can also take $V_1$ as a length two $G$-representation).
\item[(b)]
the quotient $V_2$ of $V$ has a length two subrepresentation which is generically free for $G$ (use Lemma \ref{lPassageSLtoG}).
\end{itemize}
By upper-semicontinuity of the dimension of the stabilizers, the generic stabilizer in $V_2$ is \emph{finite}. Suppose that the generic stabilizer in $V$ were nontrivial. Then, denoting by $v= (v_1, \: v_2)$ the decomposition of a vector $v\in V$ with respect to the $\mathrm{SL}_n (\CC )$-equivariant splitting $V = V_1\oplus V_2$, we get: there exists a (fixed) $v_2\in V_2$ such that for general $v_1\in V_1$ the stabilizer $G_v$ of $v = (v_1, \ v_2)$ inside $G$ is finite and nontrivial, and contained in the (fixed) finite group $G_{v_2} \subset G$. Since $\mathrm{SL}_n (\CC )$ is a reductive Levi subgroup in $G$, hence maximal reductive, and $G_{v_2}$ is finite, there is an element $t\in U =\CC^ n$ (the unipotent radical) such that $t G_{v_2} t^{-1} \subset \mathrm{SL}_n (\CC )$. Hence we get, replacing $v$ by $t \cdot v = (v_1', \: v_2 ' )$, that there is a $v_2' \in V_2$ such that for general $v_1'\in V_1$ the stabilizer of $(v_1' , \: v_2' )$ in $G$ is a nontrivial finite subgroup of $\mathrm{SL}_n (\CC )$ which contradicts the property of $V_1$ in (a) (being generically free for $\mathrm{SL}_n (\CC )$).
\end{proof}
\section{Rationality if the representation dimension is large} \xlabel{sRationalityLarge}
The aim of this section is to prove
\begin{theorem}\xlabel{tA}
Let $V$ be a generically free \emph{indecomposable} representation of the special affine group $G=\mathrm{SL}_n (\CC ) \ltimes \CC^n$. Then there
is a constant $k=k(n)$ depending on $n$ such that if $\dim V
\ge k(n)$, then $V/G$ is rational.
\end{theorem}
We consider exclusively filtrations of type (\ref{formulaFiltration2}) for $V$ in this section: recall that this is a filtration 
\begin{gather*}
0 \subset V_0' \subset V_1' \subset \dots \subset V_{l-1}' \subset V_l' =V ,
\end{gather*}
defined inductively as follows: $Q_{l-j}' := V_{l-j}' /V_{l-j-1}'$, for $j=0, \dots , l$, is the maximal completely reducible quotient
representation of $V_{l-j}'$.
\
The indecomposability assumption on $V$ cannot be dropped due to Remark \ref{rExceptionsWeCannotDo}. The proof of Theorem \ref{tA} will be preceded by
some lemmas.
\begin{lemma}\xlabel{lUseOfIndecomposability}
Fix $n$ and the filtration length $l$ of $V$. Suppose that $\tilde{V}= V/V'_{l-2}$ is an exceptional two-step extension 
\[\tilde{V}=V_1\oplus W, \: 0\to
S\to V_1
\to Q\to 0\] (notation similar to Theorem \ref{tB}, so $W$ is an $\mathrm{SL}_n (\CC )$-representation and $S\subset Q\otimes\CC^n$, $Q\subset
S\otimes (\CC^n)^{\vee }$), and
$Q\oplus W$ contains
$<n^2 -1$ copies of
$\CC$. Suppose that $V$ is indecomposable, and write $W=W_1 \oplus \dots \oplus W_k$ where the $W_i$'s are defined inductively as follows:
\begin{itemize}
\item
$W_1$ contains all irreducible summands $W'$ of $W$ such that for the $G$-spans inside $V$ we have $\langle G\cdot W' \rangle \cap \langle G \cdot Q
\rangle \neq 0$.
\item
$W_{j+1}$ contains all irreducible summands $W''$ of $W$ which are not already in $W_1, \dots , W_j$ and satisfy
\[
\langle G\cdot W'' \rangle \cap \langle G\cdot (Q+W_1 + \dots + W_j) \rangle \neq 0 \, 
\]
(so in fact
\[
\langle G\cdot W'' \rangle \cap \langle G\cdot  W_j \rangle \neq 0 \, \, ,
\]
and then, by the indecomposability of $V$, we have $W=W_1\oplus\dots \oplus W_k$).
\end{itemize}
Then for $s\in\mathbb{N}$ there is a constant $c(s)$ such that if $\dim V\ge c(s)$  then $k\ge s$, and for $\dim V$ sufficiently large, $V$ has a
generically free
$G$-quotient
$\hat{V}$ such that $V\to \hat{V}$ has fibre dimension larger than $n^2 +n-1$, so that
$V/G$ is rational. 
\end{lemma}
\begin{proof}
By Corollary \ref{cCharExceptional} we know already that $S$ and $Q$
are limited to finitely many possibilities, so we have to show the same for the $W$'s. By Propositions \ref{pStructureAffRepresentations} and
\ref{pDualStructureTheorem} we obtain inclusions of $\mathrm{SL}_n(\CC )$-representations
\begin{gather*}
\langle G\cdot Q \rangle \subset Q\otimes \mathrm{Sym}^l (\CC^{n+1}), \\
W_1 \subset Q \otimes \mathrm{Sym}^l (\CC^{n+1} ) \otimes\mathrm{Sym}^l (\CC^{n+1})^{\vee }, \\
W_2 \subset Q \otimes \left( \mathrm{Sym}^l (\CC^{n+1} ) \otimes\mathrm{Sym}^l (\CC^{n+1})^{\vee }\right)^{\otimes 2}, \\
\vdots 
\end{gather*}
and so forth, so that together with the possibilities for $Q$, also those for the $W's$ are limited. Thus for $\dim V$ to become arbitrarily large,
 we need $k$ to become very large, i.e. $W$ is highly decomposable. Here one should note that this does not necessarily imply that eventually $\tilde{V}$
is no longer exceptional: the action of $G$ on $\tilde{V}$ can be not generically free and remain so after adding arbitrary nontrivial $\mathrm{SL}_n
(\CC )$-summands to $\tilde{V}$, e.g. if $\tilde{V} = (\CC^{n+1})^{\vee }$ or, more generally, there are some nontrivial translations in the stabilizer
in general position. So we have to resort to some other type of argument here, namely we show directly that if $k$ becomes very large, then $V$ has a
generically free quotient $\hat{V}$ with large fibre dimension as claimed in the statement of the Lemma.
\
Note that there can only be finitely many trivial summands $\CC$ by hypothesis among the $W's$, so that if $k$ becomes large, 
$W=W_1\oplus \dots \oplus W_k$, will eventually be a good $\mathrm{SL}_n (\CC )$-representation. Together with $W$, 
the maximal completely reducible
subrepresentation $V_0$ of $V$ becomes large and highly reducible by the construction of the $W_i$'s. An $\mathrm{SL}_n (\CC
)$-subrepresentation $R_0$ of $V_0$ gives rise to a quotient $\hat{V}_{R_0}$ of $V$ as follows: take the $G$-span of 
$R_0^{\vee}\subset V^{\vee }$
(inclusion as
$\mathrm{SL}_n (\CC )$-representation) and  take the dual of this span. For $k$ large, we may choose $R_0$ in such a way that 
$W_1\oplus \dots \oplus
W_{\kappa }$ is contained (as $\mathrm{SL}_n (\CC )$-representation) in $\hat{V}_{R_0}$ modulo its maximal completely reducible
subrepresentation, and 
$W_1\oplus \dots \oplus W_{\kappa }$ is good, so that $\hat{V}_{R_0}$ is $G$-generically free by Lemma \ref{lPassageSLtoG} (note that
condition (b) of this Lemma will be automatically satisfied if the dimension of $W_1\oplus \dots \oplus W_{\kappa }$ is sufficiently large, using
Proposition
\ref{pStructureAffRepresentations}, since the filtration length of $\hat{V}_{R_0}$ modulo its maximal completely reducible
subrepresentation is bounded with $l$). For large
$k$, we can then also arrange that the fibre dimension of $V\to \hat{V}_{R_0}$ becomes large, since $\dim V_0$ grows with $k$, and $V_0$
becomes highly reducible. 
\end{proof}
\begin{remark}\xlabel{rWhySoComplicated}
The complicated procedure used in the proof of the previous Lemma \ref{lUseOfIndecomposability} is justified by the complexity of $G$-representations
for which we want to give some examples:
\begin{itemize}
\item
the examples of $(\CC^{n})^{\vee }\otimes (\CC^{n+1})^{\vee}$ resp. its dual show that $V$ can have reducible $V_0$ or $Q_l$ without being decomposable.
\item
Consider the subrepresentation $V\subset (\CC^{n})^{\vee }\otimes \mathrm{Sym}^2 (\CC^{n+1})^{\vee}$ with
\[
Q_0= (\CC^n)^{\vee}, \: Q_1= \mathrm{Sym}^2 (\CC^n)^ {\vee} \oplus \Lambda^2 (\CC^n)^{\vee }, \: Q_2= \mathrm{Sym}^3 (\CC^n )^{\vee } \, .
\]
Here $\Lambda^2 (\CC^n)^{\vee }$ is not in the $G$-span of $\mathrm{Sym}^3 (\CC^n )^{\vee }$ and
\[
Q_0'= (\CC^n)^{\vee}, \: Q_1'= \mathrm{Sym}^2 (\CC^n)^ {\vee} , \: Q_2'= \mathrm{Sym}^3 (\CC^n )^{\vee }\oplus \Lambda^2 (\CC^n)^{\vee } \, .
\]
\item
Consider the subrepresentation $V$ of \[(\mathrm{Sym}^3( \CC^n)^{\vee } \oplus \Sigma^{2,1}( \CC^n)^{\vee } \oplus \Lambda^3 ( \CC^n)^{\vee })
\otimes  \mathrm{Sym}^2 (\CC^{n+1})^{\vee }\] with type \ref{formulaFiltration} filtration such that
\begin{eqnarray*}
Q_0 =&\mathrm{Sym}^3( \CC^n)^{\vee } \oplus \Sigma^{2,1}( \CC^n)^{\vee } \oplus \Lambda^3 ( \CC^n)^{\vee } , \\
Q_1 =&\mathrm{Sym}^4 ( \CC^n)^{\vee }\oplus \Sigma^{3,1}( \CC^n)^{\vee }\oplus \Sigma^{2,1,1}( \CC^n)^{\vee }, \\
Q_2 =&\mathrm{Sym}^5 ( \CC^n)^{\vee }\, .
\end{eqnarray*}
Diagrammatically, we can picture which of these summands map to which ones under the translations in $\CC^n$ as follows:
\
\setlength{\unitlength}{1cm}
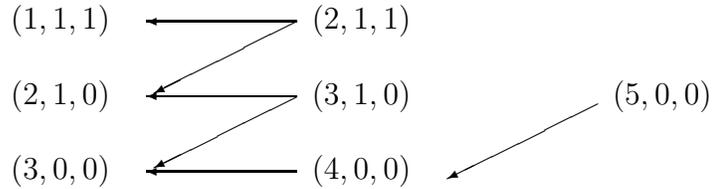
\begin{center}
\begin{picture}(4,3)
\put(5,1.5){$(5,0,0)$}
\put(1,0.5){$(4,0,0)$}
\put(1,1.5){$(3,1,0)$}
\put(1,2.5){$(2,1,1)$}
\put(-3,0.5){$(3,0,0)$}
\put(-3,1.5){$(2,1,0)$}
\put(-3,2.5){$(1,1,1)$}
\put(4.8,1.5){\vector(-2,-1){2}}
\put(0.8,0.6){\vector(-1,0){2}}
\put(0.8,1.6){\vector(-1,0){2}}
\put(0.8,2.6){\vector(-1,0){2}}
\put(0.8,1.6){\vector(-2,-1){1.9}}
\put(0.8,2.6){\vector(-2,-1){1.9}}

\end{picture}\end{center}
Thus the $G$-span $S_1$ of $\mathrm{Sym}^5 ( \CC^n)^{\vee }$ intersects the span $S_2$ of $\Sigma^{3,1}( \CC^n)^{\vee }$ nontrivially, and $S_2$
intersects the $G$-span $S_3$ of $\Sigma^{2,1,1}( \CC^n)^{\vee }$ nontrivially, but $S_1\cap S_3=0$. Moreover, the filtration of $V$ of type
\ref{formulaFiltration2} has
\begin{eqnarray*}
Q_0' =&\mathrm{Sym}^3( \CC^n)^{\vee }, \\
Q_1' =&\mathrm{Sym}^4 ( \CC^n)^{\vee }\oplus \Sigma^{2,1}( \CC^n)^{\vee } \oplus \Lambda^3 ( \CC^n)^{\vee } , \\
Q_2' =&\mathrm{Sym}^5 ( \CC^n)^{\vee }\oplus \Sigma^{3,1}( \CC^n)^{\vee }\oplus \Sigma^{2,1,1}( \CC^n)^{\vee }\, .
\end{eqnarray*}
\end{itemize}
\end{remark}
We now turn to the proof of Theorem \ref{tA}.
\begin{proof}
Recall that we use filtrations of type \ref{formulaFiltration2} throughout this proof. The Theorem is true for two-step filtrations by Theorem \ref{tB}, so we can suppose that the filtration length $l$ of $V$ satisfies $l\ge 3$.
$V$ is supposed to be generically free and indecomposable, and we will distinguish cases according to the type of two-step extension
$\tilde{V}=V/V'_{l-2}$, $\tilde{V} = V_1 \oplus W$ (as in Lemma \ref{lUseOfIndecomposability}) which $V$ has as quotient.
\begin{enumerate}
\item
For $0\to S\to V_1 \to Q \to 0$, we have that $Q\oplus W$ contains $\ge n^2-1$ copies of $\CC$. Then we obtain rationality for $V/G$ by taking a $(G,
\:\mathrm{SL}_n (\CC ))$-section as in the proof of Theorem \ref{tB}. So we suppose $Q\oplus W$ contains $< n^2-1$ copies of $\CC$ in the following.
\item
$\tilde{V}$ is not exceptional. Then $\tilde{V}/G$ is rational by Theorem \ref{tB}, and $V/G$ generically a vector bundle over it (as $G$ acts
generically freely on $\tilde{V}$), so $V/G$ is rational.
\item
$\tilde{V}$ is exceptional (and $Q\oplus W$ is assumed to contain $< n^2-1$ copies of $\CC$ by step 1). Then by Lemma \ref{lUseOfIndecomposability}, it only remains to consider the case where the filtration length $l$ becomes large. In this case we get rationality of $V/G$ from
the stable rationality of generically free $G$-representations and by Corollary \ref{cNotGenericallyFree}.
\end{enumerate}
\end{proof}


\begin{thebibliography}{999999999999}
\bibitem[Art]{Art}
Artin, M., \emph{Brauer Severi Varieties} (notes by A. Verschoren), in: Brauer Groups in Ring Theory and Algebraic Geometry, Proceedings,
Antwerp 1981, Springer LNM \textbf{917}, Springer-Verlag (1982), 194-210
\bibitem[Bogo79]{Bogo79}
Bogomolov, F.A., \emph{Holomorphic tensors and vector bundles on projective varieties}, Math. USSR Izvestija, Vol. \textbf{13}, no. 3 (1979),
499-555
\bibitem[Bogo86]{Bogo86}
Bogomolov, F., \emph{Stable rationality of quotient varieties by simply connected groups}, Mat. Sbornik \textbf{130} (1986), 3-17
\bibitem[Bo-Ka]{Bo-Ka} Bogomolov, F. \& Katsylo, P., \emph{Rationality of some quotient varieties}, Mat. Sbornik \textbf{126}
(1985), 584-589
\bibitem[Gi-Sza]{Gi-Sza}
Gille, P. \& Szamuely, T., \emph{Central Simple Algebras and Galois Cohomology}, Cambridge studies in advanced mathematics \textbf{101},
Cambridge University Press (2006)
\bibitem[Po-Vi]{Po-Vi}
Popov, V.L. \& Vinberg, E.B., \emph{Invariant Theory}, in: Algebraic Geometry IV, A.N. Parshin, I.R. Shafarevich
(eds.), Encyclopedia of Mathematical Sciences, vol. \textbf{55}, Springer Verlag (1994)
\bibitem[Ros]{Ros}
Rosenlicht, M., \emph{Some basic theorems on algebraic groups}, American Journal of Mathematics \textbf{78} no. 2, 401-443
\bibitem[Sa99]{Sa99}
Saltman, D., \emph{Lectures on Division Algebras}, CBMS Regional Conference Series in Mathematics \textbf{94}, AMS (2009)
\bibitem[Specht]{Specht}
Specht, W., \emph{Darstellungstheorie der affinen Gruppe}, Mathemati-\\sche Zeitschrift Volume \textbf{43}, Number 1/December 1938, 120-160
\end{thebibliography}
\end{document}